\documentclass[reqno]{amsart}

\usepackage{graphicx, amsmath, amssymb, amsthm, tikz}

\newtheorem{theorem}{Theorem}

\newtheorem{lemma}{Lemma}
\newtheorem{proposition}{Proposition}

\begin{document}

\title[]{A curious dynamical system in the plane}

\author[]{Stefan Steinerberger}
\address{Department of Mathematics, University of Washington, Seattle, WA 98195, USA}
\email{steinerb@uw.edu}

\author[]{Tony Zeng}
\address{Department of Mathematics, University of Washington, Seattle, WA 98195, USA}
\email{txz@uw.edu}

\begin{abstract} For any irrational $\alpha > 0$ and any initial value $z_{-1} \in \mathbb{C}$, we define
a sequence of complex numbers $(z_n)_{n=0}^{\infty}$ as follows: $z_n$ is $z_{n-1} + e^{2 \pi i \alpha n}$ or $z_{n-1} - e^{2 \pi i \alpha n}$, whichever has the smaller absolute value.
If both numbers have the same absolute value, the sequence terminates at $z_{n-1}$ but this happens rarely.
This dynamical system has astonishingly intricate behavior: the choice of signs in $z_{n-1} \pm e^{2 \pi i \alpha n}$ appears to eventually become periodic (though the period can be large). We prove that if one observes periodic signs for a sufficiently long time (depending on $z_{-1}, \alpha$), the signs remain periodic for all time. The surprising complexity of the system is illustrated through examples.
\end{abstract}

\maketitle

\section{Introduction and Results}
\subsection{Introduction}
Let $z_{-1} \in \mathbb{C}$ and let $\alpha > 0$ be an irrational number. We define a sequence of complex numbers $(z_n)_{n=0}^{\infty}$ by setting, for $n \geq 0$, 
$$ z_n = z_{n-1} \pm e^{2 \pi i \alpha n}$$
where the sign is chosen so as to minimize the absolute value. If both choices of sign lead to a complex number with the same absolute value, i.e. $|z_{n-1} +e^{2 \pi i \alpha n}| =|z_{n-1} -e^{2 \pi i \alpha n}|$, then we say the sequence is only defined up to $z_{n-1}$ (this is of minor importance as it happens very rarely). We were motivated by the remarkably intricate complicated behavior that this dynamical system is able to exhibit. Examples are shown throughout starting with Fig. 1 and Fig. 2.

\begin{center}
    \begin{figure}[h!]
\begin{tikzpicture}
    \node at (0,0) {\includegraphics[width=0.3\textwidth]{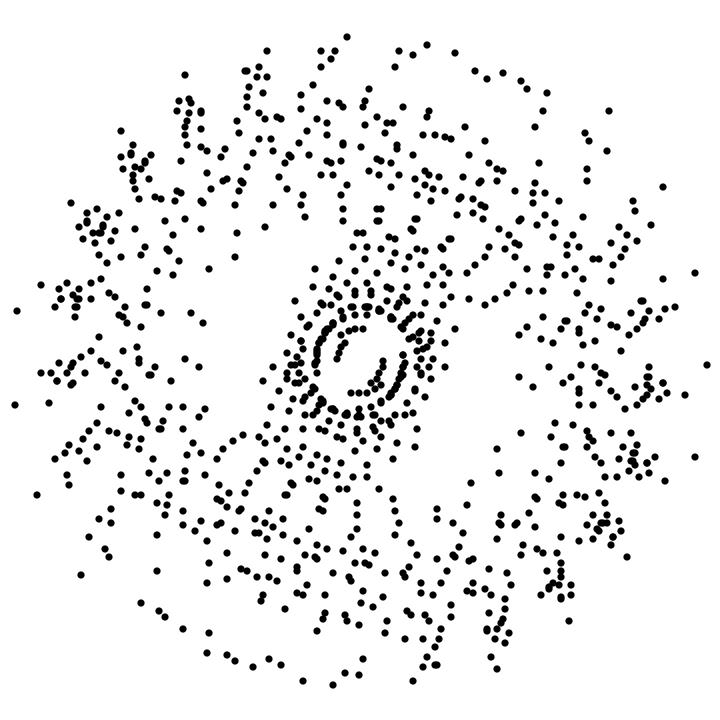}};
    \node at (4,0) {\includegraphics[width=0.3\textwidth]{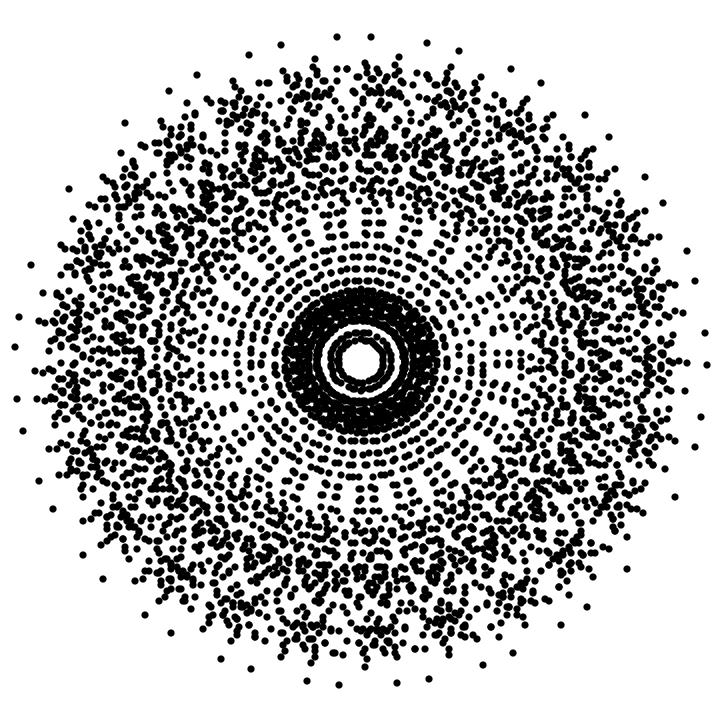}};
    \node at (8,0) {\includegraphics[width=0.3\textwidth]{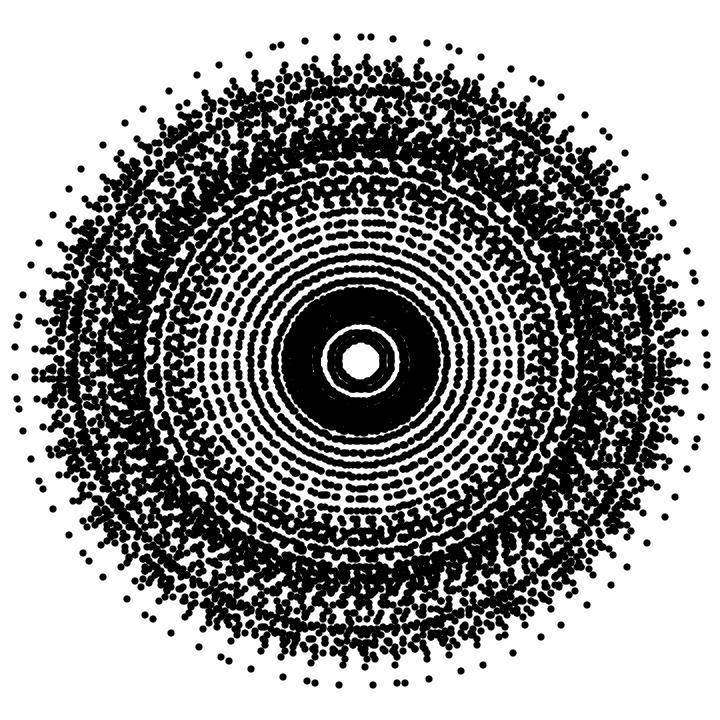}};
\end{tikzpicture}
\caption{The first 1000 (left), 5000 (middle) and 10000 (right) terms for $z_{-1} = 0.0001+ 5i$ and $\alpha =  1.0415/\sqrt{2 \pi^2}$. The sign pattern is periodic with period 222.}
    \end{figure}
\end{center}

Numerical examples suggest that `most' initial values $z_{-1}$ and `most' values of $\alpha$ lead to a sequence $(z_n)_{n=0}^{\infty}$ all of whose elements are eventually trapped in a single circle (by which we mean that all elements of the sequence $z_n$ for $n \geq N$ are contained in a circle), see Fig. 4 or Fig. 5. However, exceptions are not infrequent (especially when $\alpha$ is close to $0, 1/2$ or $1$): most of the examples shown in the various Figures were discovered by looking at many random choices for $\alpha$ and $z_{-1}$.

\begin{center}
    \begin{figure}[h!]
\begin{tikzpicture}
    \node at (0,0) {\includegraphics[width=0.3\textwidth]{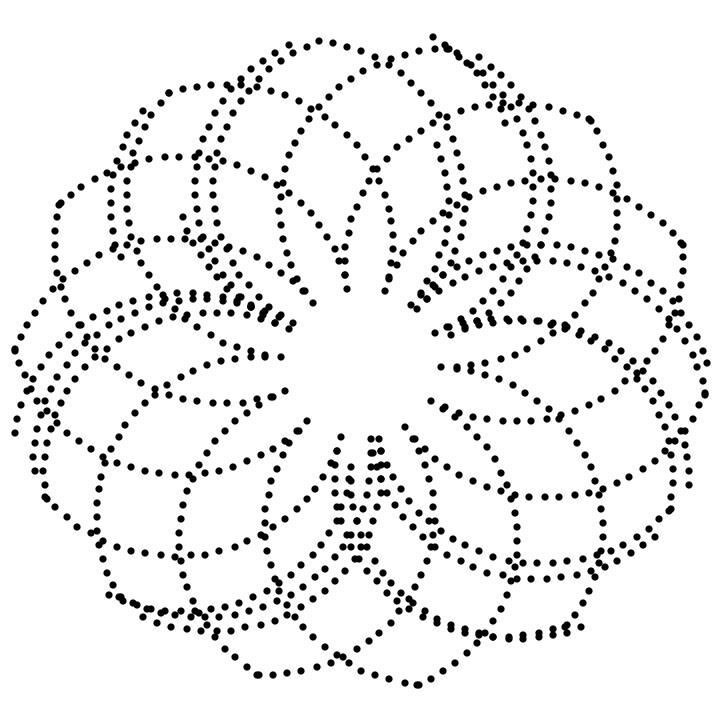}};
    \node at (4,0) {\includegraphics[width=0.3\textwidth]{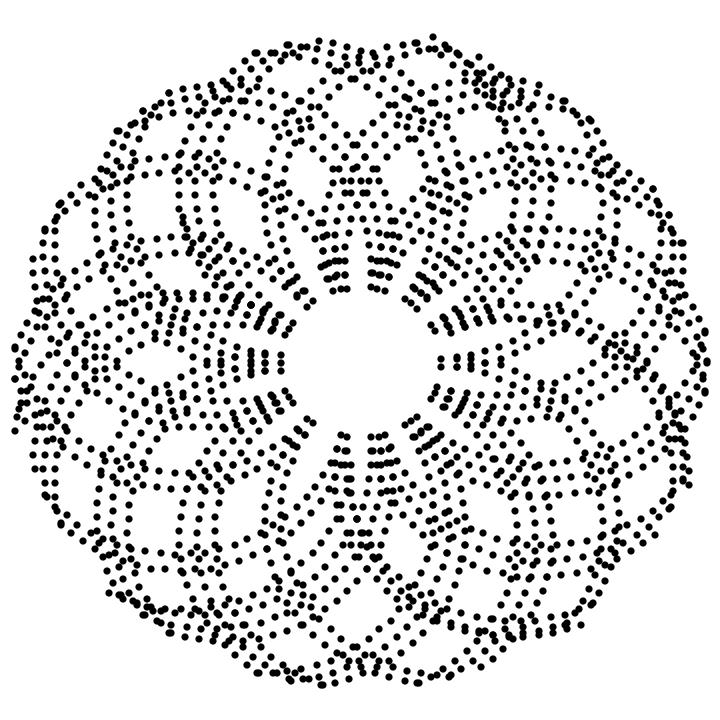}};
    \node at (8,0) {\includegraphics[width=0.3\textwidth]{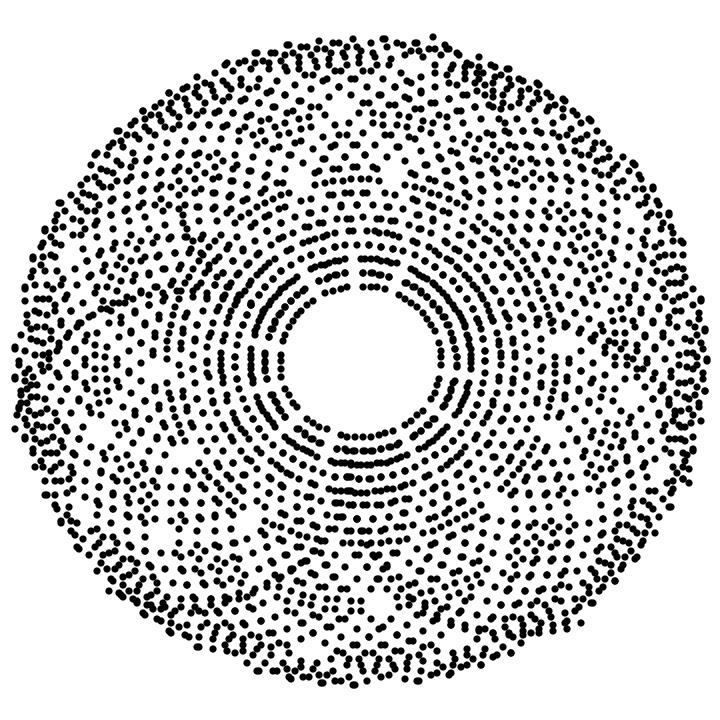}};
\end{tikzpicture}
\caption{The system with initial value
 $z_{-1} = 2.0176 + 4.8585 i$ and $\alpha = 0.00702367$. The pictures show the
terms $z_n$ for $2000 \leq n \leq 3000$ (left), $2000 \leq n \leq 4000$ (middle) and $2000 \leq n \leq 5000$ (right). The sequence of signs is periodic with period $51$.}
    \end{figure}
\end{center}

\subsection{Periodic signs implies concentric circles.} 
We start with a basic observation: when looking at the dynamical system for `generic' initial values, the sequence of signs $\pm 1$ that are chosen appears to become periodic.  A periodic choice of signs corresponds to a simple geometric pattern: the sequence is ultimately contained in a union of finitely many concentric circles (see, for example, Fig. 3).

\begin{figure}[h!]
    \centering
    \begin{tikzpicture}
        \node at (0,0) {\includegraphics[width=0.3\textwidth]{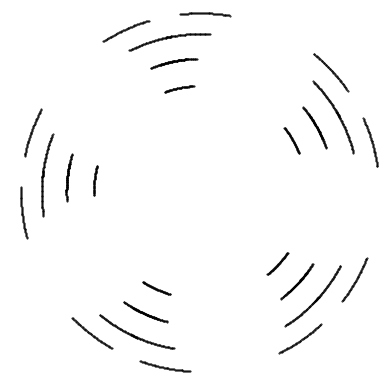}};
        \node at (4,0) {\includegraphics[width=0.3\textwidth]{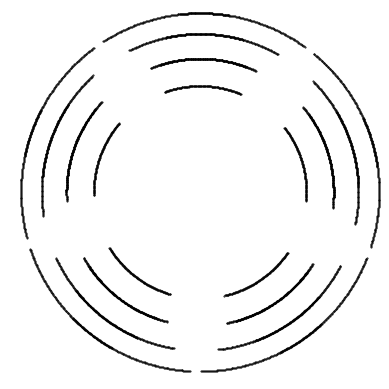}};
        \node at (8,0) {\includegraphics[width=0.3\textwidth]{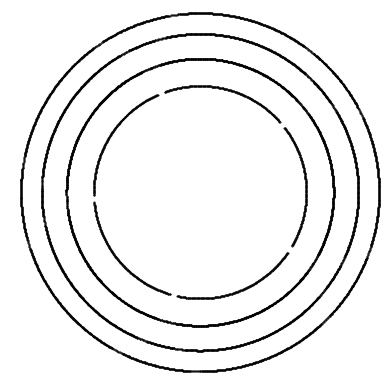}};
    \end{tikzpicture}
    \caption{The system with initial value \(z_{-1} = 1 - i\) and \(\alpha = \sqrt{2} / 3\). The pictures show the terms \(z_n\) for \(100 \le n \le 2000\) (left), \(100 \le n \le 5000\) (middle), and \(100 \le n \le 8000\) (right). The sequence of \(\pm\) signs is periodic with period \(14\).}
    \label{fig:14}
\end{figure}

\begin{theorem}
    If the sequence of signs is eventually periodic with period $p$, eventually all elements of the sequence are contained in a union of at most $p$ concentric circles.
\end{theorem}

The inequality is necessary: it is fairly common to see signs with period 2 ($+1$ and $-1$ alternating) where the associated sequence is contained in a single circle. It is not clear to us how the period of the signs is connected to the number of circles beyond the one-sided inequality: Fig. \ref{fig:14} shows an example where signs are periodic with period $14$ but all elements are ultimately contained in only $4$ concentric circles.

\subsection{Long-time behavior: a single circle.} One natural question is whether the figures shown throughout the paper are (a) numerically accurate and (b) reflect the overall long-time behavior. Two instructive examples are given in Fig. \ref{fig:16}: one observes rather intricate behavior for the first few elements of the sequence after which the sequence suddenly becomes very simple and all subsequent elements are contained in a single circle.
\begin{figure}[h!]
    \centering
    \begin{tikzpicture}
        \node at (0,0) {\includegraphics[width=0.35\textwidth]{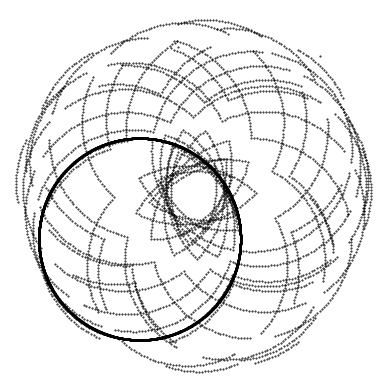}};
        \node at (6,0) {\includegraphics[width=0.37\textwidth]{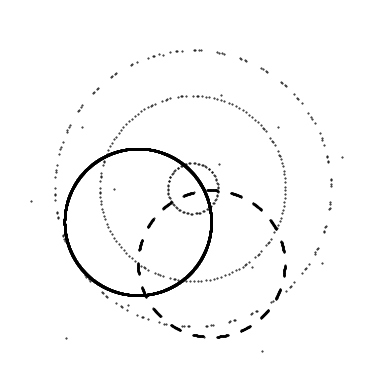}};
    \end{tikzpicture}
    \caption{Left: system with \(z_{-1} = 2+5i\) and \(\alpha = 1.1054\sqrt{2}/(2\pi)\). Right: system with \(z_{-1} = 5e^{2\pi i \cdot 0.00883}\) and \(\alpha = \sqrt{7}/(2\pi)\). Both sequences are quite intricate before settling down in a single circle.}
        \label{fig:16}
\end{figure}

This raises a natural question: when observing more complicated behavior as in Fig. 1 and Fig. 2, are we guaranteed that they remain at that level of complexity? Maybe everything always ends up in a single circle (after possibly millions and millions of terms)? Before explaining our main results, we introduce a useful concept: the ambiguity sequence. Given such a sequence $(z_n)_{n=0}^{\infty}$, the associated ambiguity sequence $(a_n)_{n=0}^{\infty}$ is the sequence of real numbers 
$$ a_n =  \left||z_{n-1} + e^{2\pi i n \alpha}| -  |z_{n-1} - e^{2\pi i n \alpha}|\right|.$$
One can think of $(a_n)_{n=0}^{\infty}$ as a measure of ambiguity that quantifies the difference between the two choices of sign: if $a_n$ is large, then one sign very clearly leads to a complex number with smaller norm. If $a_n$ is small, the difference is less pronounced. One could also think of very small values of $a_n$ as a possible sign of danger: higher numerical accuracy may be required. By assumption, we always have $a_n > 0$ since $z_n$ is not defined if both choices of sign lead to a complex number with the same absolute value. We can now state the first half of our main result. Informally, it will end up saying
\begin{quote}
    if one observes a sequence always choosing the same sign for a while and if the levels of ambiguity $a_n$ never get too small, then the sign sequence will be periodic for all time.
\end{quote}

This informal statement is in need of some additional precision. We first state the result for a special case where the signs end up being constant: in that case, all the contributing components can be made explicit.
The general periodic case (Theorem 3) will be based on Theorem 2 and uses very similar ideas.

\begin{theorem}[Single circle] Let $\alpha > 0$ be irrational. Then, using $(p_{\ell}/q_{\ell})_{\ell=1}^{\infty}$ to denote the convergents, the following holds: if, for some $k, \ell \in \mathbb{N}$ all the signs for $k \leq n \leq k+q_{\ell}$ are the same and if
$$ \min_{k \leq n \leq k + q_{\ell}} a_n \geq  \left(  4\pi +  \frac{4\pi}{|e^{2\pi i \alpha} - 1|} \right)   \frac{20}{q_{\ell}},$$
then the sign is constant for all $n \geq k$.
\end{theorem}

\textbf{Remarks.}
\begin{enumerate}
    \item Theorem 1 then implies that, in that case, all sufficient large elements of the sequence are contained in a single circle.
    \item $k$ and $\ell$ are arbitrary parameters. $k$ allows for different starting indices, $\ell$ can be seen as a measure of the scale. A larger scale requires a weaker inequality to be true for a longer time.
    \item Generically, if the sequence ends up choosing the same sign,  the ambiguity sequence is bounded away from 0. This means that, generically, Theorem 2 will eventually be able to certify periodicity (see the proof for details).
\end{enumerate}

\begin{figure}[h!]
    \centering
    \begin{tikzpicture}
        \node at (0,0) {\includegraphics[width=0.3\textwidth]{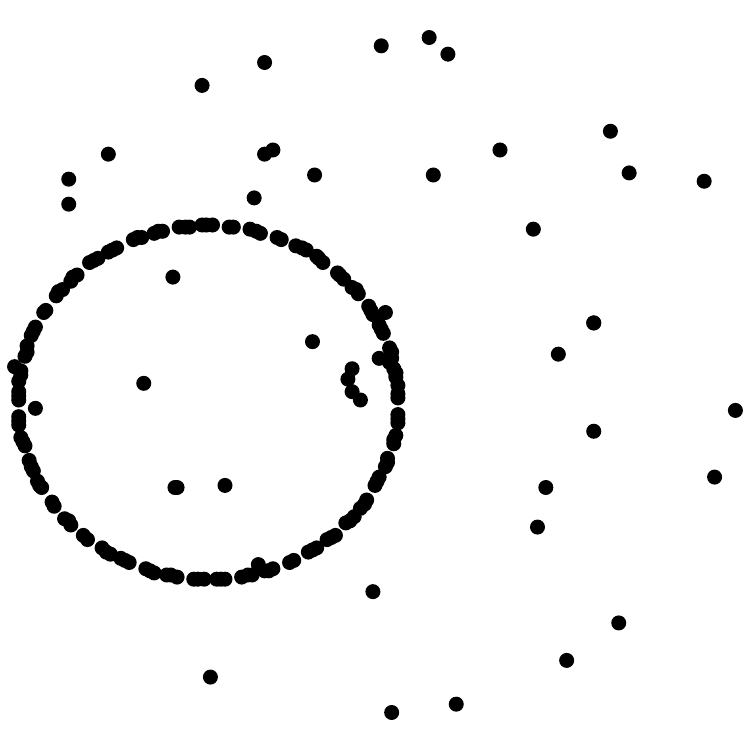}};
        \node at (6,0) {\includegraphics[width=0.5\textwidth]{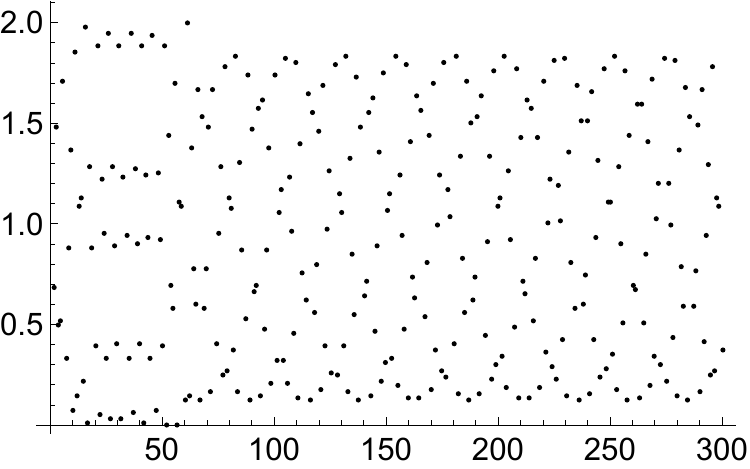}};
    \end{tikzpicture}
    \caption{Left: system with $z_{-1} = -1/2 - i$ and $\alpha=1/\sqrt{6}$. Right: the value of $a_n$ along the evolution of the system}
    \label{fig:15}
\end{figure}

We quickly illustrate Theorem 2 for a concrete example (see Fig. \ref{fig:15}): we start with $z_{-1} = -1/2 - i$ and $\alpha=1/\sqrt{6}$ and observe that the sequence initially jumps around before, around the 60th term, starting to always choose the sign $+1$ and ending up in a circle. Computing the convergents $p_{\ell}/q_{\ell}$ for $\alpha = 1/\sqrt{6}$, the sequence of denominators $q_{\ell}$ is given by
$$ 1, 2, 5, 22, 49, 218, 485, 2158, 4801, 21362, \dots$$
Picking $q_{9} = 4801$, an explicit computation shows that
$$ \left(  4\pi +  \frac{4\pi}{|e^{2\pi i \alpha} - 1|} \right) \frac{20}{4801}= 0.0796\dots$$
 Simultaneously, we observe, numerically, that $a_n \geq 0.12$ for all $100 \leq n \leq 100 + 4801$. Theorem 2 applies and guarantees that the sequence will always keep choosing the sign $+1$ and remain in a single cycle.
 There is a little bit of flexibility in how to choose $k$ and $q_{\ell}$: the sequence of ambiguities satisfies $a_n \geq 0.1$ for all $n \geq 100$. The appearance of convergent fractions $(p_{\ell}/q_{\ell})_{\ell=1}^{\infty}$ is related to the assumption of $\alpha$ being irrational: if $\alpha$ is irrational but extremely close to a rational number $p/q$, then $e^{2\pi i n \alpha}$ will be extremely close to a $q-$periodic sequence over extremely large periods of time. This is captured by the growth of $q_{\ell}$. For example, if $\alpha = 10/17 + 10^{-7} \sqrt{2}$ then the sequence of $q_{\ell}$ is $1, 2, 5, 17, 415944, \dots$ which shows that the ambiguity may need to be checked for a large number of consecutive terms.

\subsection{Long-time behavior: general periodicity.}
Arguably the most interesting case is not covered by Theorem 2: examples where the sequence exhibits an interesting periodic pattern as seen in Figure 1 and Figure 2. We can now state our main result: if one observes periodic signs for a sufficiently long period of time, then the signs remain periodic for all time.

\begin{theorem} Let $\alpha > 0$ be irrational. There exists an (explicit) function 
$$g:(0,1) \times \mathbb{N} \rightarrow \mathbb{N}$$
depending on $\alpha$ and $z_{-1}$ such that for all $\eta \in (0,1)$ the following holds: if
\begin{enumerate}
    \item the sequence of signs is $p-$periodic for $k \leq n \leq k + g(\eta, p)$ and
    \item the ambiguity sequence is not too small in that range
    $$ \min_{k \leq n \leq k + g(\eta, p)}  \left||z_{n-1} + e^{2\pi i n \alpha}| -  |z_{n-1} - e^{2\pi i n \alpha}|\right| \geq \eta,$$
\end{enumerate}
then the sequence of signs is $p-$periodic for all $n \geq k$.
\end{theorem}

\textbf{Remarks.}
\begin{enumerate}
\item Theorem 3 can be summarized as: if one observes periodic signs for a while and does not run into issues of numerical accuracy, the sequence of signs is going to be periodic for all time.
    \item As in Theorem 2, the result is nearly sharp in the sense that for `generic' examples with periodic signs the ambiguity sequence remains bounded away from 0 and Theorem 3 applies for a suitable choice of $k$ and $\eta$.
    \item The function $g(\eta, p)$ could be made explicit in terms of $p$ and $\alpha$: as in Theorem 2, the convergents coming from the continued fraction expansion play a role in how quickly $g(\eta, p)$ tends to infinity as $\eta \rightarrow 0^+$ (see proof).
\end{enumerate}

\subsection{Stability and Symmetries}
We start by describing an important notion of stability: if we have a sequence starting $z_{-1}$ and the minimal ambiguity is uniformly bounded away from 0, then changing $z_{-1}$ a little does not substantially change the behavior of the sequence.

\begin{proposition}
    Let $\alpha > 0$ and $z_{-1} \in \mathbb{C}$. Replacing \(z_{-1}\) by \(z_{-1} + w\) with
 $$  |w| < \frac{1}{2} \inf_{n \ge 0} a_n = \frac{1}{2} \inf_{n \geq 0} \big| |z_n + e^{2\pi i n \alpha}| - |z_n - e^{2\pi i n \alpha}| \big| $$
shifts the entire sequence by \(w\), i.e.
$z'_n = z_n + w$ for all $n \geq 0$.
\end{proposition} 
Proposition 1 naturally suggests that the minimal ambiguity along a sequence should be a natural way to classify the behavior of sequences according to their starting value $z_{-1}$: if the minimal ambiguity is large, then nearby starting values are going to behave very much like the sequence arising from $z_{-1}$ does. Small values of the minimal ambiguity suggest that a small change in the initial value $z_{-1}$ can lead to very different dynamical behavior. Plotting the ambiguity function
$$ z_{-1} \rightarrow \inf_{n \geq 0} \big| |z_n + e^{2\pi i n \alpha}| - |z_n - e^{2\pi i n \alpha}| \big|$$
shows the most surprising pictures (see Figure \ref{fig:amb_normal} and Figure \ref{fig:amb_chaos}). The dependency of \(\inf_{n \geq 0} a_n\) on \(z_{-1}\) seems to be highly nontrivial. 

\begin{figure}[h!]
    \centering
    \begin{tikzpicture}
        \node at (0,0) {\includegraphics[scale=0.47]{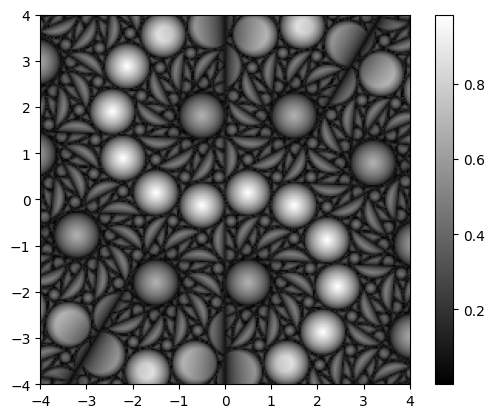}};
        \node at (6,0) {\includegraphics[scale=0.47]{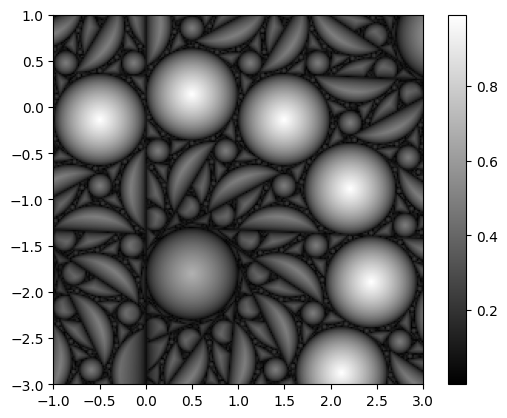}};
        \node at (0,-5) {\includegraphics[scale=0.47]{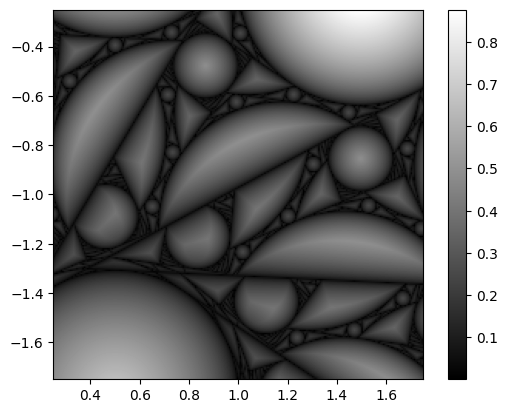}};
        \node at (6,-5) {\includegraphics[scale=0.47]{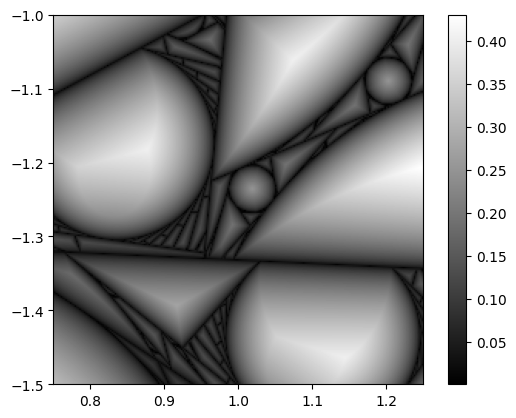}};
    \end{tikzpicture}
    \caption{With \(\alpha = \sqrt{2}\), (square roots of) minimum ambiguities of the first one thousand terms of sequences plotted against \(z_{-1}\).}
    \label{fig:amb_normal}
\end{figure}

Fig. \ref{fig:amb_normal} suggests the presence of symmetries: we note three such symmetries.

\begin{proposition}
    Let \(z_n\) be given by $0 < \alpha < 1$ and $z_{-1}$, with sign sequence \(\varepsilon_n\). 
    \begin{enumerate}
        \item The sequence arising from $1-\alpha$ is  \( \left(\overline{z_n}\right)\), with signs \(\varepsilon_n\).
        \item The sequence arising from \(\alpha \pm 0.5 \pmod 1\) is \(\left(z_n\right)\) with signs \(\varepsilon_n \cdot (-1)^n\).
      \item The sequence arising from \(-z_{-1}\) is \(-z_n\), with signs \(-\varepsilon_n\).
    \end{enumerate}
    \label{prop:sym}
\end{proposition}

While not appearing to be self-similar in the way classical fractals are, Fig. \ref{fig:amb_chaos} shows that several layers of self-similarity are possible. Fig. \ref{fig:amb_normal} appears to feature balls (as well as balls cut by lines). Balls show up naturally as part of the dynamical system: there are two disks with radius $1/2$ whose interior, when used as initial values, lead to a sequence with constant signs.

\begin{proposition}
    Let $\alpha \in (0, 1)$. If
  $$  \left|z_{-1} - \frac{1}{1 - e^{2\pi i \alpha}} \right| < \frac{1}{2},$$
    then the arising sequence always picks sign $-1$ (and ends up in a circle). Likewise, if
  $  \left|z_{-1} - 1/(e^{2\pi i \alpha}-1) \right| < 1/2$,
the sequence always pick sign $+1$.
\end{proposition}

These two disks (which are clearly visible in Fig. \ref{fig:amb_normal} and Fig. \ref{fig:amb_chaos}) naturally have to be disjoint; this is reflected in the identity
$$\forall~\alpha \notin \mathbb{N} \qquad \mbox{Re}~\frac{1}{e^{2\pi i \alpha}-1} = \frac{1}{2}.$$
 Other disks in the figures seem to appear for similar reasons: characterizing them all seems to be non-trivial and potentially an interesting area for further research.
 
 \begin{figure}[h!]
    \centering
    \begin{tikzpicture}
        \node at (0,0) {\includegraphics[scale=0.45]{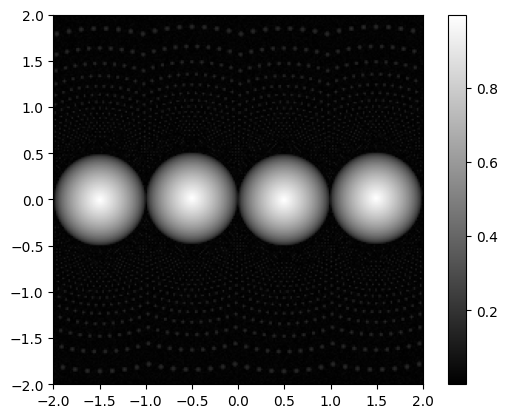}};
        \node at (6,0) {\includegraphics[scale=0.45]{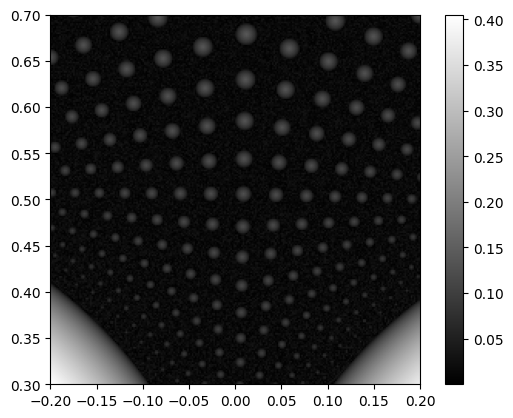}};
        \node at (0,-5) {\includegraphics[scale=0.45]{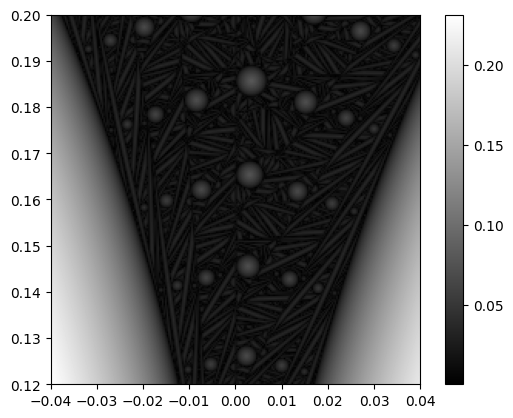}};
        \node at (6,-5) {\includegraphics[scale=0.45]{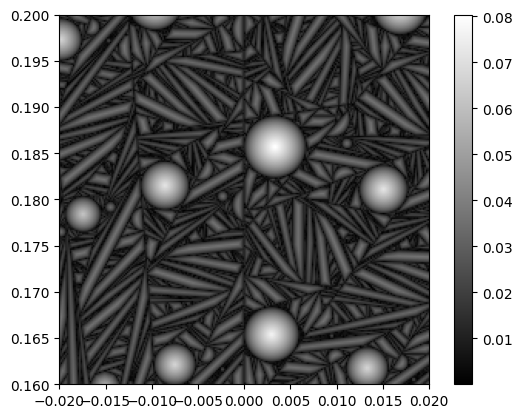}};
    \end{tikzpicture}
    \caption{With \(\alpha = 0.5 + \sqrt{3} / 300\), (square roots of) minimum ambiguities of the first ten thousand terms plotted against \(z_{-1}\).}
    \label{fig:amb_chaos}
\end{figure}

 The next natural question is whether the non-circular mysterious shapes in Fig. \ref{fig:amb_normal} and Fig. \ref{fig:amb_chaos} can be explained and this will be covered by the next result.

\begin{theorem}
    Let $\alpha$ be irrational and let $z_{-1} \in \mathbb{C}$ lead to a sequence whose sign pattern is eventually periodic. The set of initial values that eventually exhibit the same (periodic) sign pattern is contained in a finite union of a finite intersection of disks and half spaces.
\end{theorem}

The proof shows slightly more: the interior of such regions is given by an intersection of (open) disks and half spaces. If we use
 $z_{-1} \rightarrow \inf_{n \geq 0} \big| |z_n + e^{2\pi i n \alpha}| - |z_n - e^{2\pi i n \alpha}| \big|$
to partition $\mathbb{C}$ into regions with similar dynamical behavior (separated by the zeros of that function) and if every sequence is eventually periodic, then Theorem 4 would have some more implications: the plane would be tiled by intersections of disks and half spaces (up to a set of measure 0).

\subsection{Open problems} While we have a relatively good understanding of the case where one observes periodic sign patterns, the inverse problem is completely open.
\begin{quote}
    \textbf{Problem.} Is the sign pattern always eventually periodic?
\end{quote}
We do not know how difficult that problem is. For fixed $\alpha$, Proposition 1 shows that the set of initial values $z_{-1}$ that lead to periodic orbits tend to have a certain `openness' condition in terms of the ambiguity sequence: this might be an indication that hypothetical non-periodic orbits, if they exist at all, must be rare. A second natural question concerns the length of periods. 

\begin{quote}
    \textbf{Long periods.} If a sequence eventually stabilizes into a periodic choice of signs, how long can this period be depending on $\alpha$?
\end{quote}

\begin{figure}[h!]
     \centering
     \begin{tikzpicture}
         \node at (0,0) {\includegraphics[scale=0.3]{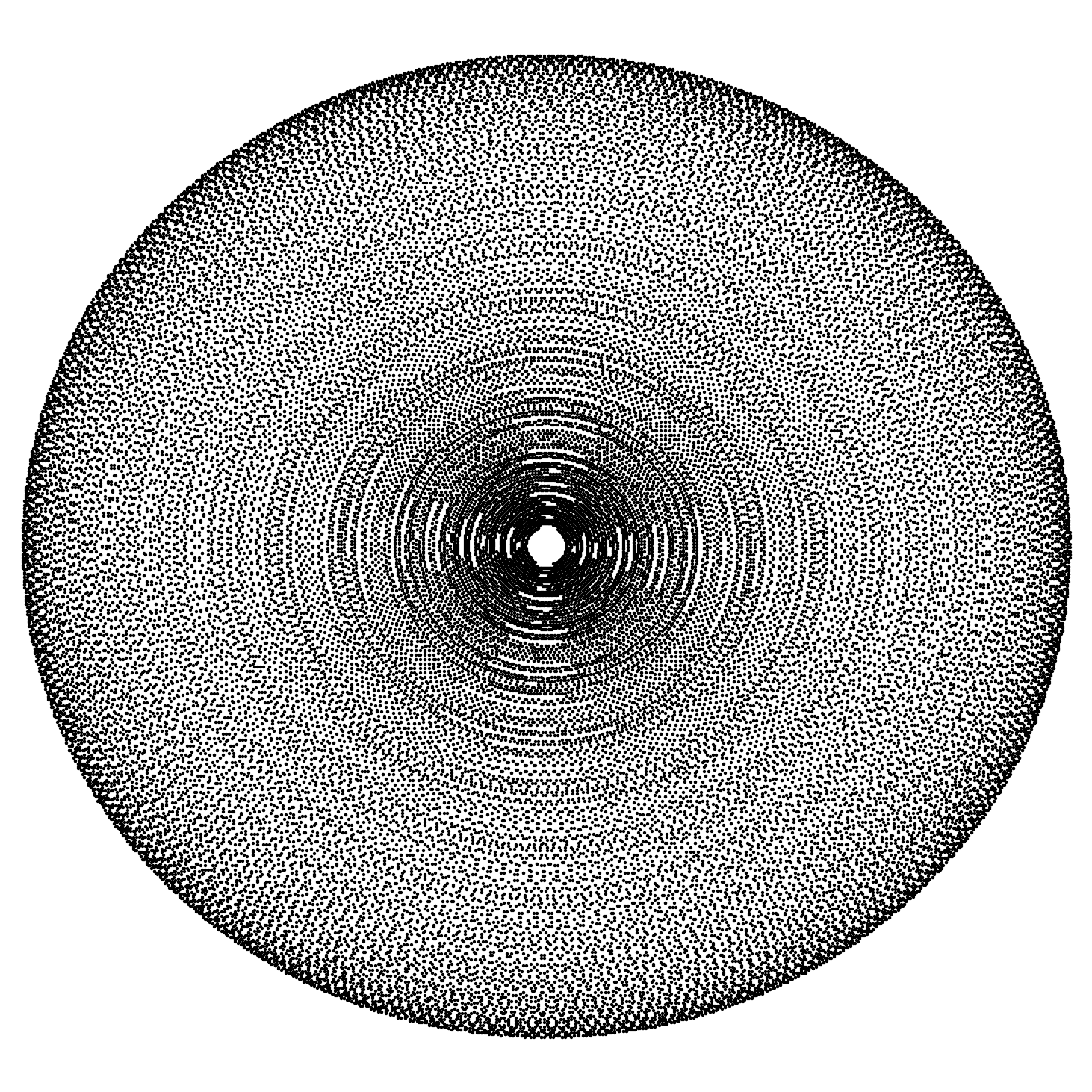}};
         \end{tikzpicture}
     \caption{A sequence with period 874 (first million terms, only every 20th term displayed).}
     \label{fig:874}
 \end{figure}

It seems that sequences can have arbitrarily long periods. A nice example is a sequence with period 874 given by 
\[ \alpha = 0.5010866 \quad \mbox{and}  \quad\quad z_{-1} = 0.747467+0.445271i\]
whose associated ambiguity sequence is bounded below by $6 \cdot 10^{-6}$. The example is shown in Fig. \ref{fig:874}. Moreover, it appears as if having a large period requires $\alpha \mod 1$ to be close to $0,1/2,1$ and one might be tempted to ask whether, for some universal $c>0$, an inequality along the lines of
$$ \mbox{length of period} \lesssim \frac{c}{ \|2 \alpha\|}$$
might be true (where $\|x\| = \inf_{n \in \mathbb{Z}} \left| x - n \right|$ is the distance to the nearest integer). It is not clear to us how difficult to this question is. It would be interesting to see whether there is a way to construct sequences with arbitrarily large periods: we observe, as a consequence of the symmetries, the following period-doubling trick.

\begin{proposition}
    Let \((z_n)\)  given by \(\alpha\) and \(z_{-1}\). If the sign sequence \(\varepsilon_n\) is \(p\)-periodic with \(p\) odd, the sequence with initial point \(z_{-1}\) and \(\alpha + 0.5 \pmod 1\) is \(2p\)-periodic.
\end{proposition}

The largest period we are currently aware of is a sequence with period 2258 given by $\alpha = 0.50015827$ and initial value 
$$ z_{-1} = 0.5761982862055985+0.9356408428886818i.$$
However, we have no reason to believe that this is in any way extremal; it appears that, by choosing $\alpha$ close to $1/2$, arbitrarily large periods are possible.
 Another innocent question is as follows.

\begin{quote}
    \textbf{Which numbers are periods?} Which numbers can arise as periods of the dynamical system for some $\alpha$ and $z_{-1}$?
\end{quote}
It appears, judging purely from empirical observations, as if not all integers can arise as a period. For example, we have observed empirically that periods are not divisible by 4. Is this always the case?

\subsection{Related results.} 
 Our original interest was sparked by a beautiful result of 
Bettin-Molteni-Sanna \cite{Bettin} who prove, among many other things, the following: pick some $x \in \mathbb{R}$, let $x_1 = 1$ and
$$ x_{n} = \begin{cases} x_{n-1} + \frac{1}{n} \qquad &\mbox{if}~x_n \leq x \\ x_{n-1} - \frac{1}{n} \qquad &\mbox{if}~a_n > x. \end{cases}$$
This corresponds to taking the standard harmonic series $\sum 1/n$ and inserting signs in a greedy way to get as close as possible to $x$. It is easy to see that $|x_n - x| \leq 2/n$ for $n$ sufficiently large. Bettin-Molteni-Sanna \cite{Bettin} show the existence of subsequence that converge \textit{much} faster, even faster than a polynomial rate in $n$, for generic $x$. We discovered our system by looking at hypothetical two-dimensional analogues of this result.
A number of (vaguely) related results come from the abstract theory of rearrangements of series in Banach spaces (see Kadets-Kadets \cite{banach}), for example the results of Calabi-Dvoretzky \cite{calabi} or Dvoretzky-Hanani \cite{dv}. After completion of the project we learned that our results can be thought of as results in the area concerned with the \textit{dynamics of piecewise isometries}. Its setting is as follows: given a subset $X \subset \mathbb{R}^d$ and a finite partition $X = P_1 \cup P_2 \cup \dots P_r$, a map $T:X \rightarrow X$ is a piecewise isometry if is restriction to $P_i$ is an isometry for all $1 \leq i \leq r$. 

\begin{center}
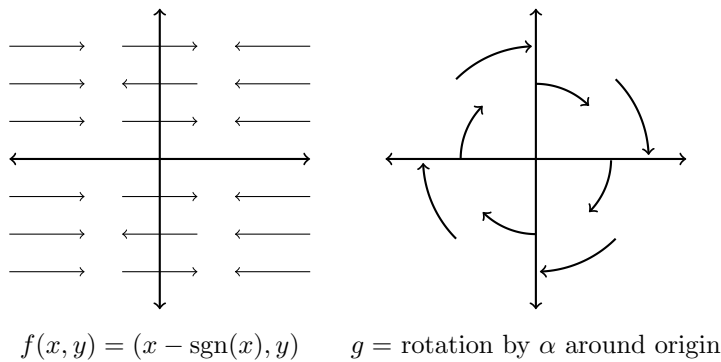
\begin{figure}[h!]
    \begin{tikzpicture}
   \draw [thick, <->] (-2,0) -- (2,0);
  \draw [thick, <->] (0,-2) -- (0,2);
  \draw [->] (2, 1) -- (1,1);
    \draw [->] (2, -1) -- (1,-1);
  \draw [->] (2, -1.5) -- (1,-1.5);
    \draw [->] (2, 1.5) -- (1,1.5);
        \draw [->] (2, 0.5) -- (1,0.5);
    \draw [->] (2, -0.5) -- (1,-0.5);
   \draw [->] (0.5, 1) -- (-0.5,1);
  \draw [->] (-0.5, 0.5) -- (0.5,0.5);
     \draw [<-] (0.5, 1.5) -- (-0.5,1.5);
        \draw [->] (0.5, -1) -- (-0.5,-1);
  \draw [->] (-0.5, -0.5) -- (0.5,-0.5);
     \draw [<-] (0.5, -1.5) -- (-0.5,-1.5);
  \draw [->] (-2, 1) -- (-1,1);
    \draw [->] (-2, -1) -- (-1,-1);
  \draw [->] (-2, -1.5) -- (-1,-1.5);
    \draw [->] (-2, 1.5) -- (-1,1.5);
 \draw [->] (-2, 0.5) -- (-1,0.5);    
   \draw [->] (-2, -0.5) -- (-1,-0.5);    
   \node at (0,-2.5) {$f(x,y) = (x - \mbox{sgn}(x), y)$};
   \draw [thick, <->] (3,0) -- (7,0);
  \draw [thick, <->] (5,-2) -- (5,2);
\draw [thick,domain=180:135, ->] plot ({5+cos(\x)}, {sin(\x)});
\draw [thick,domain=135:92, ->] plot ({5+1.5*cos(\x)}, {1.5*sin(\x)});
\draw [thick,domain=90:45, ->] plot ({5+cos(\x)}, {sin(\x)});
\draw [thick,domain=45:2, ->] plot ({5+1.5*cos(\x)}, {1.5*sin(\x)});
\draw [thick,domain=359:315, ->] plot ({5+cos(\x)}, {sin(\x)});
\draw [thick,domain=315:272, ->] plot ({5+1.5*cos(\x)}, {1.5*sin(\x)});
\draw [thick,domain=270:225, ->] plot ({5+cos(\x)}, {sin(\x)});
\draw [thick,domain=225:182, ->] plot ({5+1.5*cos(\x)}, {1.5*sin(\x)});
\node at (5,-2.5) {$g =$ rotation by $\alpha$ around origin};
    \end{tikzpicture}
    \caption{ $g \circ f$ is a piecewise isometric map whose orbits, suitably rotated, correspond to orbits of our procedure.}
    \label{fig:piecewise}
\end{figure}
\end{center}

Many of the fundamental results in the area have been pioneered by Arek Goetz, starting with his PhD thesis \cite{goetz0} and the subsequent work \cite{goetz1, goetz000, goetz00, goetz2, goetz11} as well as
Ashwin-Goetz \cite{ash1,ash2},  Ashwin-Goetz-Peres-Rodrigues \cite{ashwin}, Boshernitzan-Goetz \cite{misha}, 
Cheung-Goetz-Qyas \cite{Cheung} and Goetz-Quas \cite{quas}.
There have been many subsequent developments and applications, we refer to Bruin-Deane \cite{bruin}, Deane \cite{deane}, Park et al \cite{park}, Smith et al \cite{smith} and Sturman \cite{stur}. At first glance, our definition can be thought of as a piecewise isometric map with respect to half-spaces, defined by $e^{2\pi i n \alpha}$, that rotate. By introducing an additional rotation, we obtain a piecewise isometric map (see Fig. 9) that is given as the composition of a translation on two half-spaces and a global rotation. Its iterates $y_n$, when rotated by $e^{2\pi i n\alpha}$, correspond to iterates $x_n$ of our procedure. This type of piecewise isometric map seems to not have been studied before and it stands to reason that our greedy sign choice perspective may arguable be easier to work with. Moreover, it appears that our main results (Theorem 2 and Theorem 3) appear to be new within the context of the dynamics of piecewise isometric maps.

\section{Proofs}

\subsection{Proof of Theorem 1}
\begin{proof} Introducing the sign sequence $(\varepsilon_n)_{n=0}^{\infty}$, we have for all $n > N$ that
    \[ z_n = z_N + \sum_{j = N + 1}^n \varepsilon_j e^{2\pi i \alpha j}. \]
Assume the sign sequence is periodic with period $p$, then, for all $\ell \in \mathbb{N}$,
    \begin{align*}
    \sum_{j=k+(\ell -1)p+1}^{k+\ell p} \varepsilon_j e^{2\pi i j \alpha} &=\sum_{j=k+1}^{k+p} \varepsilon_{j+ \ell p} e^{2\pi i (j+\ell p) \alpha} \\
    &= \sum_{j=k+1}^{k+p} \varepsilon_{j} e^{2\pi i j \alpha}  e^{2\pi i \ell p \alpha} = e^{2 \pi i  \ell \alpha p } \sum_{j=k+1}^{k+p} \varepsilon_j e^{2\pi i j \alpha}.    
    \end{align*}
    This last sum is independent of $\ell$: it only depends on the period $p$ and the starting point $k$ (this observation will occur repeatedly in subsequent proofs). Abbreviating
    \begin{equation} \label{eq:rec}
           q_k = \sum_{j=k+1}^{k+p} \varepsilon_j e^{2\pi i j \alpha}, 
    \end{equation}
    we deduce that for \(k \ge N\) and \(\ell \ge 0\),
    \begin{align*} 
         z_{k + \ell p} &= z_k + \sum_{j=k+1}^{k+\ell p}  \varepsilon_j e^{2\pi i j \alpha} = z_k + \sum_{s = 0}^{\ell-1} \sum_{j= k + s p + 1}^{k + (s+1)p} \varepsilon_j e^{2\pi i j \alpha} \\
         &= z_k + q_k\sum_{s=0}^{\ell -1} e^{2 \pi i \alpha p s} = z_k + q_k \frac{e^{2\pi i \alpha p \ell}-1}{e^{2\pi i \alpha p}-1} \\
         &= \left[z_k - \frac{q_k}{e^{2\pi i \alpha p} - 1}\right] + q_k \frac{e^{2\pi i \alpha p \ell}}{e^{2\pi i \alpha p}-1}
    \end{align*}
    For \(\ell \in \mathbb{N}\), this subsequence of points lies on a circle with center
    \[ c_k = z_k - \frac{q_k}{e^{2\pi i \alpha p} - 1}. \]
    Repeating the above argument for \(k \in \{N, \ldots, N + p - 1\}\) shows that all \(z_n\) for \(n > N\) lie on one of \(p\) circles. It remains to prove that these circles have a common center.
    By using the definition of the sequence and iterating the recursion \eqref{eq:rec}
    \begin{align*}
        c_{k + 1} &= z_{k + 1} - \frac{q_{k + 1}}{e^{2\pi i \alpha p} - 1} \\ 
        &= z_k + \varepsilon_{k + 1}e^{2\pi i \alpha(k + 1)} - \frac{q_k + \varepsilon_{k + p + 1}e^{2\pi i \alpha(k + p + 1)} - \varepsilon_{k + 1}e^{2\pi i \alpha p(k + 1)}}{e^{2\pi i \alpha p} - 1} \\
        &= z_k - \frac{q_k}{e^{2\pi i \alpha p} - 1} = c_k.
    \end{align*}
Therefore all the circles share the same center.
\end{proof}

\subsection{A Sampling Lemma}
The purpose of this section is to prove a Lemma that will be useful in the proof. We were unable to find the Lemma in the literature but all its ingredients are standard and it is presumably fair to say that the Lemma is well-known `in spirit'.
We recall that if $\alpha$ is an irrational number, then it has an infinite continued fraction expansion: truncating this expansion leads to a rational number $p_{\ell}/q_{\ell}$ that is very close to $\alpha$. 

\begin{lemma} Suppose $\alpha$ is an irrational real with convergents $(p_{\ell}/q_{\ell})_{\ell=1}^{\infty}$. If
 $$f:\left\{z \in \mathbb{C}: |z| = 1\right\} \rightarrow \mathbb{R}$$
 is Lipschitz-continuous with Lipschitz constant $L>0$ and $k, \ell \in \mathbb{N}$ is such that
 \begin{equation} \label{eq:diamond}
      \min_{k \leq n \leq k + q_{\ell}} f(e^{2 \pi i n \alpha}) > \frac{20 L}{q_{\ell}}, 
 \end{equation}
 then
 $$ \min_{0 \leq t \leq 1} f(e^{2 \pi i t}) > 0.$$
Conversely, if $ \min_{0 \leq t \leq 1} f(e^{ 2 \pi i t}) > 0$, then there exists $\ell_0 \in \mathbb{N}$ such that \eqref{eq:diamond} holds for all $k \geq 0$ and all $\ell \geq \ell_0$.
\end{lemma}
\begin{proof} We first show that it suffices to prove the statement for $k=0$. For arbitrary $k \geq 0$, we can define $g:\left\{z \in \mathbb{C}: |z| = 1\right\} \rightarrow \mathbb{R}$ via
$$ g(z) = f( e^{2 \pi i k \alpha}z)$$
which corresponds to a rotation and does not change the Lipschitz constant. Then
$$  \min_{k \leq n \leq k + q_{\ell}} f(e^{2 \pi i n \alpha}) =  \min_{0 \leq n \leq  q_{\ell}} g(e^{2 \pi i n \alpha})$$
and we see that one implies the other. We will now argue that $n \alpha$, when interpreted modulo 1, is rather close to a set of equispaced points provided $n$ is chosen at the correct scale (suggested by the denominators of the convergents). The crucial ingredient is
(see \cite{bugeaud}) the inequality
$$ \left| \alpha - \frac{p_{\ell}}{q_{\ell}} \right| \leq \frac{1}{q_{\ell}^2}.$$
Therefore, for all $1 \leq n \leq q_{\ell}$, we have
$$ \left| n\alpha - \frac{p_{\ell} n}{q_{\ell}} \right| \leq \frac{n}{q_{\ell}^2} \leq \frac{1}{q_{\ell}}.$$
Since $p_{\ell}$ and $q_{\ell}$ are always coprime (see \cite{bugeaud}), we have
$$ \left\{  \frac{p_{\ell} n}{q_{\ell}}~\mbox{mod}~1:1 \leq n \leq q_{\ell} \right\} =  \left\{  \frac{n}{q_{\ell}}:0 \leq n \leq q_{\ell} -1\right\}.$$
Let now $0 \leq t \leq 1$ be arbitrary. Then there exists $1 \leq n \leq q_k$ such that
$$ \left| \left(\frac{p_{\ell} n}{q_{\ell}} \mod 1\right) - t \right| \leq \frac{1}{q_{\ell}}.$$
We deduce that 
$ \left| \left(n \alpha \mod 1\right) - t \right| \leq 3/q_{\ell}$
 and
$$ \left| \left(2 \pi n \alpha \mod 2\pi \right) - 2 \pi t \right| \leq \frac{2 \pi \cdot 3}{q_{\ell}} \leq \frac{19}{q_{\ell}}.$$
Using Lipschitz-continuity, we have
\begin{align*}
    f(e^{2\pi i t}) &\geq f(e^{2\pi i n \alpha}) - L \left| 2\pi n \alpha - 2 \pi t\right| \\
    &\geq \frac{20 L}{q_k} - \frac{19}{q_k} \cdot L \geq \frac{L}{q_k} > 0.
\end{align*}
The second part of the statement is easy: $f$, if positive, is a continuous function on a compact interval and thus bounded from below by a positive constant: the sequence $q_{\ell}$ is unbounded and the conclusion follows.
\end{proof}

\subsection{Proof of Theorem 2.}
\begin{proof}
    We assume $\alpha , z_{-1}$ are fixed.  We also suppose that, starting at $z_k$, the chosen signs are constant for a large number of iterations (we assume at least until $z_N$ where $N > k$) and that w.l.o.g. the sign that ends up being chosen is always $+$ (the case where the sign is always $-$ is completely analogous). Then, for $k < n \leq N$  
\begin{align*}
        z_n &= z_k + \sum_{j=k+1}^{n} e^{2\pi i j \alpha} = z_k +\frac{e^{2\pi i \alpha (n+1)}  - e^{2\pi i (k+1) \alpha}  }{e^{2\pi i \alpha} -1}.
\end{align*}
Introducing the map $h_k:\mathbb{S}^1 \rightarrow \mathbb{C}$
$$ h_k(z) = z_k +\frac{z  - e^{2\pi i (k+1) \alpha}  }{e^{2\pi i \alpha} -1} $$
 we have, for $k < n \leq N$,
\begin{align} \label{eq:one}
     z_n = h_k(e^{2 \pi i \alpha (n+1)}).
\end{align}

We would like to guarantee that, for all $n > k$,
$$ \left|  z_n + e^{2\pi i \alpha (n+1)} \right| <   \left|  z_n - e^{2\pi i \alpha (n+1)} \right|$$
since that would then enforce that the sign chosen at each point is always $+$. Using \eqref{eq:one},
this is equivalent to
$$ \left|    h_k(e^{2 \pi i \alpha (n+1)}) + e^{2\pi i \alpha (n+1)} \right| <   \left|   h_k(e^{2 \pi i \alpha (n+1)}) - e^{2\pi i \alpha (n+1)} \right|.$$
This inequality would be implied by the truth of the inequality
$$ \forall~|z| = 1 \qquad  \left|    h_k(z) + z\right| <   \left|   h_k(z) - z \right|.$$
Every time we pick another element of the sequence, we certify the truth of that inequality at another point on $\mathbb{S}^1$. Moreover, the sequence $e^{2\pi i \alpha (n+1)}$ is dense in $\mathbb{S}^1$.
We apply Lemma 1 to to the function
$$f(e^{2\pi i t}) = \left|   h_k(e^{2\pi i t}) - e^{2\pi i t} \right| - \left|    h_k(e^{2\pi i t}) + e^{2\pi i t}\right|.$$
Strict positivity of $f$ would guarantee stability of the sign choice for all time.
Observing that
$$ \left| \frac{d}{dt} h_k(e^{2 \pi it}) \right| = \left| \frac{d}{dt} \frac{e^{2 \pi it} }{e^{2\pi i \alpha} -1} \right| = \frac{2 \pi }{|e^{2\pi i \alpha} - 1|}$$
we note that $f$ is Lipschitz continuous with Lipschitz constant satisfying 
 $$ L \leq 4\pi +  \frac{4\pi}{|e^{2\pi i \alpha} - 1|}.$$
At this point, we can invoke Lemma 1. The desired stability is implied by the existence of a convergent $p_k/q_k$ such that
 $$ \min_{1 \leq n \leq q_k} f(e^{2 \pi i n \alpha}) > \frac{20}{q_k} \left(  4\pi +  \frac{4\pi}{|e^{2\pi i \alpha} - 1|} \right)$$
which is the desired statement. 
\end{proof}
\textbf{Remark.} The truth of the inequality
$$ \forall~|z| = 1 \qquad  \left|    h_k(z) + z\right| <   \left|   h_k(z) - z \right|$$
is necessary but not sufficient for the sign choice to be stable. However, it is \textit{almost} necessary in a way that renders Theorem 1 applicable for virtually all choices of initial data. We will now explain this in greater detail. First, suppose that there exists $|z_*| = 1$ such that the inequality \textit{strictly fails}
 $$   \left|    h_k(z_*) + z_*\right| >   \left|   h_k(z_*) - z_* \right|,$$
then there exists an entire open interval $J \subset \mathbb{S}^1 \subset \mathbb{C}$ containing $z_*$ for which the inequality fails. Since $e^{2\pi i \alpha n}$ is dense on $\mathbb{S}^1$, at some point in the future we are bound to encounter an index $n$ where one would end up choosing the $-$ sign. In that case, the statement we are trying to prove is actually false. This means that the only remaining case is
$$ \forall~|z| = 1 \qquad  \left|    h_k(z) + z\right| \leq   \left|   h_k(z) - z \right|$$
with equality for some $z_*$. This case is subtle because the relevant question is then whether an equation of the type $z_* = e^{2\pi i \alpha n}$ has a solution in $n \in \mathbb{N}$: if it does, then there exists a point at which the sequence becomes undefined. If not, then the $+$ sign is stable. However, distinguishing between the case of strict inequality and strict failure from a finite number of samples is impossible. In summary, we expect Theorem 1 to be applicable in all practically arising cases.\\

\textbf{Remark.} We also note that the inequality
$$ \forall~|z| = 1 \qquad  \left|    h_k(z) + z\right| <   \left|   h_k(z) - z \right|$$
is the type of statement that is frequently used in complex analysis to deduce that two functions have the same number of roots inside the unit disk (via Rouch\'e's theorem). It is not clear to us whether this is a coincidence.

\subsection{Proof of Theorem 3}
\begin{proof}
The main idea behind the proof is to show that the periodic setting can be reduced to the setting of Theorem 2 by arguing along a subsequence. Let us again assume that $\alpha > 0$ is a given irrational number. We assume that, starting at $z_k$, the arising sequence of signs is $p-$periodic where $p \in \mathbb{N}$ at least in the range $k \leq n \leq N$ where $N$ is large. Then, for $k \leq n \leq N$ we have
$$        z_n = z_k + \sum_{j=k+1}^{n} \varepsilon_j e^{2\pi i j \alpha},$$
where $\varepsilon_j \in \left\{-1,1\right\}$ and the sequence of signs $(\varepsilon_j)_{j \in \mathbb{N}}$ is $p-$periodic for $k \leq j \leq N$. This implies, just as in the proof of Theorem 1, that
\begin{align*}
\sum_{j=k+(\ell -1)p+1}^{k+\ell p} \varepsilon_j e^{2\pi i j \alpha} &=\sum_{j=k+1}^{k+p} \varepsilon_{j+ \ell p} e^{2\pi i (j+\ell p) \alpha} \\
&= \sum_{j=k+1}^{k+p} \varepsilon_{j} e^{2\pi i j \alpha}  e^{2\pi i \ell p \alpha} = e^{2 \pi i  \ell \alpha p } \sum_{j=k+1}^{k+p} \varepsilon_j e^{2\pi i j \alpha}.    
\end{align*}
Observe that this last sum is independent of $\ell$: it only depends on the period $p$ and the starting point $k$. Abbreviating
$$ q = \sum_{j=k+1}^{k+p} \varepsilon_j e^{2\pi i j \alpha},$$
we deduce
\begin{align} \label{eq:2}
     z_{k + \ell p} &= z_k + \sum_{j=k+1}^{k+\ell p}  \varepsilon_j e^{2\pi i j \alpha} = z_k + \sum_{s = 0}^{\ell-1} \sum_{j= k + s p + 1}^{k + (s+1)p} \varepsilon_j e^{2\pi i j \alpha} \\
     &= z_k + q\sum_{s=0}^{\ell -1} e^{2 \pi i \alpha p s} = z_k + q \frac{e^{2\pi i \alpha p \ell}-1}{e^{2\pi i \alpha p}-1}. \nonumber
\end{align}
By definition of the sign sequence $(\varepsilon_{n})_{n=0}^{\infty}$,
$$ z_{k+\ell p +1} = z_{k + \ell p} + \varepsilon_{k + \ell p + 1} e^{2 \pi i (k+\ell p +1)\alpha}$$
which implies (since that sign has been chosen so as to minimize the absolute value) 
\begin{equation} \label{eq:3}
\left| z_{k + \ell p} + \varepsilon_{k + \ell p + 1} e^{2 \pi i (k+\ell p +1)\alpha} \right| < \left| z_{k + \ell p} - \varepsilon_{k + \ell p + 1} e^{2 \pi i (k+\ell p +1)\alpha} \right|.
\end{equation}
The next step consists of observing that  $ \varepsilon_{k + \ell p + 1}$ is independent of $p$ (by $p$-periodicity for small indices). Thus
$$\varepsilon = \varepsilon_{k + \ell p + 1} e^{2 \pi i (k+1)\alpha}$$
is independent of $\ell$ as long as $k \leq k + \ell p + 1 \leq N$. Plugging \eqref{eq:2} into \eqref{eq:3}, we observe that,
 as long as $k \leq k + \ell p + 1 \leq N$,
$$ \left| z_k + q \frac{e^{2\pi i \alpha p \ell}-1}{e^{2\pi i \alpha p}-1} + \varepsilon  e^{2\pi i \ell p \alpha}\right| <  \left| z_k + q \frac{e^{2\pi i \alpha p \ell}-1}{e^{2\pi i \alpha p}-1} - \varepsilon  e^{2\pi i \ell p \alpha}\right|.$$
It would be nice if that inequality were true for all $\ell$ sufficiently large and, much like in the proof of Theorem 2, we see that this would be implied by the truth of the inequality
$$ \forall ~t \in \mathbb{R} \qquad \left| z_k + q \frac{e^{it }-1}{e^{2\pi i \alpha p}-1} + \varepsilon  e^{it}\right| <  \left| z_k + q \frac{e^{it}-1}{e^{2\pi i \alpha p}-1} - \varepsilon  e^{it}\right|.$$
Applying the Sampling Lemma now, just as in the proof of Theorem 2, shows that if one observes strict (quantitative) inequality for a sufficiently long time, we can deduce the truth of the inequality. The choice of the subsequence was completely arbitrary and thus, by repeating the argument $p-1$ more times, we get the desired description.
\end{proof}

\subsection{Proof of Proposition 1.}
\begin{proof}
    Note that if \(z'_{n} = z_{n} + w\) with \(|w| < m/2\), then, by the triangle inequality,
    \[ |z_n + a_n| - |w| < |z_n + w + a_n| < |z_n + a_n| + |w|, \]
    as well as
    \[ |z_n - a_n| - |w| < |z_n + w - a_n| < |z_n - a_n| + |w|. \]
    Therefore,
 \begin{align*}
     |z_n + a_n| - |z_n - a_n| - m &< |z_n + w + a_n| - |z_n + w - a_n| \\
     &< |z_n + a_n| - |z_n - a_n| + m. 
 \end{align*} 
    Put another way, \(|z_n + w + a_n| - |z_n + w - a_n|\) is not more than \(m\) away from \(|z_n + a_n| - |z_n - a_n|\). Since \(|z_n + a_n| - |z_n - a_n|\) is at least \(m\) in magnitude, it must be the same sign as \(|z_n + w + a_n| - |z_n + w - a_n|\), which means that \(z'_{n + 1} = z_{n + 1} + w\). We assume that \(z'_{-1} = z_{-1} + w\), so by induction, we have the desired result.
\end{proof}

\subsection{Proof of Proposition 2}
\begin{proof}
 Note that \[ e^{2\pi i (1 - \alpha_0)n} = e^{2\pi in} \cdot e^{-2\pi i \alpha_0 n} = \overline{e^{2\pi i \alpha_0 n}}. \]  Since both the initial point and the additive sequence $ e^{2\pi i (1 - \alpha_0)n}$ are entirely conjugated, and the Euclidean norm is unaffected by conjugation, the new sign sequence must be identical to the original. Therefore, the resulting sequence must also be entirely conjugated. 
 Note that \[ e^{2\pi i (\alpha_0 \pm 0.5)n} = e^{2\pi i \alpha_0 n} \cdot e^{\pm \pi i n} = e^{2\pi i \alpha_0n} \cdot (-1)^n. \]
 Since the initial point is unchanged and the additive sequence $e^{2\pi i (\alpha_0 \pm 0.5)n}$ is only different by a sign, at each step of the sequence, the same two choices are presented, giving the same values of \(z_n\). In particular, because the additive sequence alternates in sign from the original, the new sign sequence must also alternate in sign relative to the original.
 If \(|z + w| < |z - w|\), then of course \(|-z - w| < |-z + w|\). Then by induction we obtain the third statement immediately. 
\end{proof}

 \subsection{Proof of Proposition 3}
The proof of Proposition 3 will make use of a simple Lemma that will also be  used in the proof of Theorem 4. We present it separately.

\begin{lemma}
Let $A,B,C \in \mathbb{C}$ with $C \neq 0$. The set of complex numbers $z_{-1}$ with the property that for all $|z|=1$
$$ \left| z_{-1} + A - z B - Cz\right| < \left|z_{-1} +A  - zB \right|$$
is an open ball of radius $-|C|/2 - \emph{Re}~B \overline{C}/|C|$ centered at $-A$. 
\end{lemma}
\begin{proof}
    For $x,y \in \mathbb{C}$, the inequality $|x+y| < |x|$ implies, after squaring and using that $|z|^2 = z \overline{z}$, that
$$  \mbox{Re}~ x \overline{y} < - \frac{|y|^2}{2}.$$
We apply this inequality with $x = z_{-1} + A - z B$ and $y = -Cz$
Since $y = -Cz$ and $|z|=1$, we have $-|y|^2/2 = -|C|^2/2$. This inequality can be written as
\[ \mbox{Re}~ \left[( z_{-1} + A - z B) (-\overline{Cz}) \right] < - \frac{|C|^2}{2}. \]
Using $(-z)(-\overline{z}) = 1$, this is
\[ \mbox{Re}~ \left[ (z_{-1} + A)(-\overline{Cz}) \right] < -\frac{|C|^2}{2} - \mbox{Re}~ B\overline{C} \]
This inequality needs to hold for all $|z|=1$, thus this is equivalent to
\[ \max_{|z|=1} \quad \mbox{Re}~ \left[ (z_{-1} + A)(-\overline{Cz}) \right] < -\frac{|C|^2}{2} - \mbox{Re}~ B\overline{C} \]
We note that since $z$ ranges over all complex numbers on the unit circle, so does $-\overline{z}$. We now argue that for any complex
$w = |w| e^{i \phi}$, one has
$$ \max_{0 \leq t \leq 2\pi} \mbox{Re}\left[ w e^{it} \right] = \max_{0 \leq t \leq 2\pi} \mbox{Re}\left[ |w| e^{i \phi} e^{it} \right] = |w|.$$
Thus the inequality turns into
$$ |z_{-1} + A|\cdot|C| < -\frac{|C|^2}{2} - \mbox{Re}~B\overline{C}$$
or, equivalently,
$$ |z_{-1} + A| < -\frac{|C|}{2} - \mbox{Re}~B \frac{\overline{C}}{|C|}.$$
\end{proof}

\begin{proof} We start with the case where all the signs are $-1$, the other case will be very similar.
    Suppose $z_{-1} \in \mathbb{C}$ and the sequence always chooses the sign $-1$. Then
    $$ z_n = z_{-1} - \sum_{k=0}^{n} e^{2\pi i \alpha k} = z_{-1} - \frac{e^{2\pi i \alpha (n+1)} - 1}{e^{2\pi i \alpha} -1}.$$
    Since the sign is always $-1$, for all $n \geq -1$
    $$ |z_n - e^{2\pi i \alpha (n+1)}| <  |z_n + e^{2\pi i \alpha (n+1)}|.$$
    Plugging in the formula for $z_n$, we deduce $A_n < B_n$ where
        $$ A_n = \left|z_{-1} + \frac{1}{e^{2\pi i \alpha}-1}  - e^{2\pi i \alpha (n+1)} \left( \frac{1}{e^{2\pi i \alpha}-1} + 1\right)\right| $$
      and
        $$ B_n = \left|z_{-1} + \frac{1}{e^{2\pi i \alpha}-1}  - e^{2\pi i \alpha (n+1)} \left( \frac{1}{e^{2\pi i \alpha}-1} - 1\right)\right|.$$        
$\alpha$ is irrational and $e^{2 \pi i \alpha n}$ is dense. By density, the truth of $A_n < B_n$ for all $n \geq N$ requires, for all $|z| = 1$, the inequality $C \leq D$ where
    \begin{align*}
        C &=  \left|z_{-1} + \frac{1}{e^{2\pi i \alpha}-1}  - z\left( \frac{1}{e^{2\pi i \alpha}-1} - 1\right) - 2z\right|\\
        D &= \left|z_{-1} + \frac{1}{e^{2\pi i \alpha}-1}  - z \left( \frac{1}{e^{2\pi i \alpha}-1} - 1\right)\right|
    \end{align*}
Appealing to the Lemma, the set of $z_{-1}$ for which this holds is a disk of radius $ -\mbox{Re} (1/(e^{2 \pi i \alpha} -1)$ with center \(1 / (1 - e^{2\pi i \alpha})\). This real part can be computed via
\begin{align*}
 \mbox{Re}~ \frac{1}{1-e^{2\pi i \alpha}} &= \mbox{Re}~ \frac{1}{1 - \cos{(2 \pi \alpha)} - i \sin{(2\pi \alpha)}} \\
 &= \mbox{Re} ~\frac{1 - \cos{(2 \pi \alpha)} + i \sin{(2\pi \alpha)}}{2 - 2\cos{(2\pi \alpha)}) } = \frac{1}{2}
\end{align*}
This is the desired claim.
Let us now suppose the sign is always $+1$. The argument is virtually identical and after repeating the same computation with flipped signs, we arrive at
$$\left|z_{-1} - \frac{1}{e^{2\pi i \alpha}-1} \right| \leq \frac{1}{2}.$$
\end{proof}

\subsection{Proof of Proposition 4.}
\begin{proof}
    By Proposition \ref{prop:sym} (b), we know that our sign sequence is given by \(\varepsilon_n \cdot (-1)^n\). Since \(\varepsilon_n\) is \(p\)-periodic and \((-1)^n\) is 2-periodic, we certainly have that the period of \(\varepsilon_n \cdot (-1)^n\) is a factor of \(2p\). In particular,
    \[ \varepsilon_n \cdot (-1)^n = \varepsilon_{n + 2p} \cdot (-1)^{n + 2p} \]
    for all \(n\). Notice that period is not a factor of 2 or \(p\), as
    \[ \varepsilon_n \cdot (-1)^n =  \varepsilon_{n + 2} \cdot (-1)^{n + 2} \Rightarrow \varepsilon_n = \varepsilon_{n + 2} \]
    cannot be true for all \(n\), and
    \[ \varepsilon_n \cdot (-1)^n = \varepsilon_{n + p} \cdot (-1)^{n + p} \Rightarrow (-1)^n = (-1)^{n + 1} \]
    is a contradiction. Thus, we see that our period is indeed \(2p\).
\end{proof}

\subsection{Proof of Theorem 4}
\begin{proof}
    Suppose the sequence $(z_n)$ has signs that are $p-$periodic starting at index $z_k$. Then, recalling the proof of Theorem 3 and with the choice of
    $$ q = \sum_{j=k+1}^{k+p} \varepsilon_j e^{2\pi i j \alpha},$$
we have, for all $\ell \in \mathbb{N}$, that the subsequence $z_{k+\ell p}$ is given by
\begin{align*} 
     z_{k + \ell p} = z_k + q\sum_{s=0}^{\ell -1} e^{2 \pi i \alpha p s} = z_k + q \frac{e^{2\pi i \alpha p \ell}-1}{e^{2\pi i \alpha p}-1}.
\end{align*}
By definition of the signs $\varepsilon_n \in \left\{-1,1\right\}$
$$ z_{n+1} = z_n + \varepsilon_{n+1} e^{2\pi i(n+1) \alpha}$$
we have
$$ z_{k+\ell p +1} = z_{k + \ell p} + \varepsilon_{k + \ell p + 1} e^{2 \pi i (k+\ell p +1)\alpha}.$$
Since the signs are chosen so as to minimize the absolute value
\begin{align*}
\left| z_{k + \ell p} + \varepsilon_{k + \ell p + 1} e^{2 \pi i (k+\ell p +1)\alpha} \right| < \left| z_{k + \ell p} - \varepsilon_{k + \ell p + 1} e^{2 \pi i (k+\ell p +1)\alpha} \right|.
\end{align*}
Again, as in the proof of Theorem 3, we recall that  $ \varepsilon_{k + \ell p + 1}$ is independent of $\ell$ because the signs are, by assumption $p-$periodic and that thus
$$\varepsilon = \varepsilon_{k + \ell p + 1} e^{2 \pi i (k+1)\alpha} \qquad \mbox{is independent of}~\ell.$$ 
Hence, for all $\ell$ and with that choice of $\varepsilon$,
\begin{equation} \label{eq:4}
 \left| z_k + q \frac{e^{2\pi i \alpha p \ell}-1}{e^{2\pi i \alpha p}-1} + \varepsilon  e^{2\pi i \ell p \alpha}\right| <  \left| z_k + q \frac{e^{2\pi i \alpha p \ell}-1}{e^{2\pi i \alpha p}-1} - \varepsilon  e^{2\pi i \ell p \alpha}\right|.
 \end{equation}
Since $\alpha$ is irrational, so is $\alpha p$ which forces $e^{2 \pi i \alpha p \ell}$ to be dense as a sequence of $\ell$ and we can thus deduce that 
\begin{equation} \label{eq:5}
\forall ~t \in \mathbb{R} \qquad \left| z_k + q \frac{e^{it }-1}{e^{2\pi i \alpha p}-1} + \varepsilon  e^{it}\right| \leq  \left| z_k + q \frac{e^{it}-1}{e^{2\pi i \alpha p}-1} - \varepsilon  e^{it}\right|.
\end{equation}
This is exactly the type of inequality covered by Lemma 2 and we deduce that $z_k$ is contained in a ball centered around some point $y_k$ and with a certain radius $r_k$
$$ z_k \in B(y_k, r_k).$$
The very same argument can now be carried out for $z_{k+1}$ for which we deduce that
$z_{k+1} \in B(y_{k+1}, r_{k+1})$ and so on. After we have completed one period and arrive at $z_{k+p}$, we stop getting any new information: one could run the argument for $z_{k+p}$ and one would arrive
at $z_{k+p} \in B(y_{k+p}, r_{k+p})$ but one could alternatively just rerun the argument for $z_k$ and force $\ell \geq 2$. By periodicity, one would arrive at Equation \eqref{eq:4} with an additional restriction on the choice of $\ell$. However, since Equation \eqref{eq:4} is implied by Equation \eqref{eq:5}, the restriction does not play any further role. We are therefore interested in the set
$$ \left\{z_{-1} \in \mathbb{C}: \forall ~ k \leq n \leq k + p-1: z_n \in B(y_n, r_n) \right\}.$$
At this point, we argue backwards: suppose $X \subset \mathbb{C}$ and $z_n \in X$. What does this tell us about $z_{n-1}$? We always have $z_{n} = z_{n-1} \pm e^{2 \pi i \alpha n}$. For a given point $w \in \mathbb{C}$, the regions of $z$ for which $|z+w| < |z-w|$ and $|z+w| > |z-w|$ are two half-spaces $H_1$ and $H_2$ and the region where $|z+w| = |z-w|$ is a line going through the origin and separating the two half spaces. Thus, the relevant `pre-image', meaning the set of $z_{n-1}$ which, under the procedure, would be mapped to $z_n$ is given by 
$$ \left\{z \in H_1: z+w \in X \right\} \cup \left\{z \in H_2: z-w \in X \right\},$$
where \(w = e^{2\pi i \alpha n}\), and \(H_1, H_2\) are defined with respect to this \(w\). This can also be written as  $\left( H_1 \cap (X-w) \right) \cup \left( H_2 \cap (X+w) \right)$. If $X$ can be written as a finite union of finite intersections of disks and half-spaces, then the same is true for $X-w$ (by translation invariance) and the same is true for $H_1 \cap (X-w)$ which is just one more intersection with a half space. The same argument applies to $H_2 \cap (X+w)$. Taking then an additional intersection with $B(y_{n-1}, r_{n-1})$ does not change the desired property either and we can induct backwards to $z_{-1}$.
\end{proof}

\textbf{Acknowledgments.} We are grateful to Anthony Quas for telling us about the area of piecewise isometries and we are indebted to Arek Goetz for extensive discussions and many helpful remarks.

\end{document}